\documentclass{amsart}

\newtheorem{theorem}{Theorem}[section]
\newtheorem{lemma}[theorem]{Lemma}

\theoremstyle{definition}
\newtheorem{definition}[theorem]{Definition}

\theoremstyle{remark}

\numberwithin{equation}{section}

\usepackage{graphicx}

\begin{document}

\title{A class of sequence spaces defined  by $l$-fractional difference operator}


\author{Sanjay Kumar mahto}
\address{Department of Mathematics, Indian Institute of Technology Kharagpur, Kharagpur 721302, India}
\curraddr{}
\email{kumarmahtosanjay@iitkgp.ac.in, skmahto0777@gmail.com}
\thanks{}

\author{P. D. Srivastava}
\address{Department of Mathematics, Indian Institute of Technology Kharagpur, Kharagpur 721302, India}
\curraddr{}
\email{ pds@maths.iitkgp.ac.in}
\thanks{}



\date{}

\dedicatory{}

\begin{abstract}
In this paper, we generalize the fractional order difference operator using $l$- Pochhammer symbol and define $l$- fractional difference operator. The $l$- fractional difference operator is further used to introduce a class of difference sequence spaces. Some topological properties and duals of the newly defined spaces are studied.
\end{abstract}

\maketitle

\section{Introduction}
Fractional order difference operators defined on the set $w$ of all sequences of complex numbers have recently got an attention to many researchers because of its applicability in the Numerical analysis \cite{baliarsingh2013unifying}, statistical convergence \cite{kadak2017weighted}, approximation theory \cite{kadakgeneralized} etc. Fractional order difference operators are basically generalizations of $m$- th order difference operators, where $m$ is a nonnegative integer.  For a proper fraction $a$, Baliarsingh in \cite{baliarsingh2013unifying} and Dutta and Baliarsingh in \cite{dutta2014note} introduced some fractional order difference operators $\Delta^{a}$, $\Delta^{(a)}$, $\Delta^{-a}$ and $\Delta^{(-a)}$. The fractional order difference operator $\Delta^{(a)}: w \rightarrow w$ is defined by
\begin{equation}\label{baliarsinghfractional}
\Delta^{(a)}x = \Big\{\sum\limits_{i=0}^{\infty}(-1)^i \frac{\Gamma(a + 1)}{i! \Gamma(a - i + 1)}x_{k-i}\Big\}_k
\end{equation}
for all $x = \{x_k\} \in w$. Assuming $x_{-s} = 0$ for all positive integers $s$, the fractional order difference operator (\ref{baliarsinghfractional}) can be restated in terms of Pochhammer symbol as
\begin{equation}\label{baliarsinghfractional2}
\Delta^{(a)}x = \Big\{\sum\limits_{i=0}^{k}\frac{(-a)_i}{i!}x_{k-i}\Big\}_k,
\end{equation}
where the Pochhammer symbol $(a)_k$ is defined by
\begin{equation}
(a)_{k} =
\begin{cases}
1,  & \text{when} \    k = 0 \\
a(a +1)(a + 2)(a + 3) \dots (a + k-1), & \text{when} \ k \in \mathbb{N}\setminus \{0\}
\end{cases}
\end{equation}\par
A sequence space is a linear subspace of the set $w$ of all sequences of complex numbers. Difference operators are used to construct sequence spaces which are called difference sequence spaces. Kizmaz \cite{kizmaz1981certain} was the first to develop the idea of difference sequence space. For the spaces $ \mu = l_{\infty}, c$ and $c_0$, he defined and studied some Banach spaces
\begin{equation}\label{eq:kizmaz}
\mu(\Delta)
= \left\{ x =\{x_k\} \in w : \Delta x \in\mu\right\},
\end{equation}
 where $\Delta x = \{x_k - x_{k+1}\}$. Et and Colak \cite{et1995some} replaced the first order difference operator by an $m$-th order difference operator in $\mu(\Delta)$ and defined  $BK$- spaces
 \begin{equation}\label{eq:colak}
 \mu(\Delta^{m}) = \left\{ x =\{x_k\} \in w: \Delta^{m} x \in \mu\right\}.
 \end{equation}
 Later on, many authors such as Mursaleen \cite{mursaleen1996generalized}, Et et al. \cite{et2004some}, Aydin and Ba\c{s}ar \cite{aydin2004some}, Isik \cite{icsik2010generalized}, Srivastava and Kumar \cite{srivastava2010generalized}, Maji and Srivastava \cite{maji2014b}, Bhardwaj and Bala \cite{bhardwaj2010generalized}, Djolovi\'{c} and Malkowsky \cite{djolovic2011characterizations}, Chandra and Tripathy \cite{chandra2002generalised}, Hazarika and Savas \cite{hazarika2011some}, Hazarika \cite{hazarika2014difference} etc. have generalized (\ref{eq:kizmaz}) and (\ref{eq:colak}) and studied other difference sequence spaces. Recently, fractional difference operators have been used to define fractional difference sequence spaces such as in \cite{baliarsingh2013some}, \cite{dutta2014note}, \cite{kadak2015generalized} and many more.\par
 One way to enhance the fractional order difference operator (\ref{baliarsinghfractional2}) is to apply $l$-Pochhammer symbol \cite{diaz2007hypergeometric} instead of applying simply Pochhammer symbol. The $l$-Pochhammer symbol $(a)_{k,l}$, for a complex number $a$ and a real number $l$, is introduced by Diaz and Pariguan in \cite{diaz2007hypergeometric} as follows:
\begin{equation}
(a)_{k,l} =
\begin{cases}
1, & \text{when} \   k = 0 \\
a(a + l)(a + 2l)(a + 3l) \dots (a + (k-1)l), & \text{when} \ k \in \mathbb{N}\setminus \{0\}
\end{cases}
\end{equation}

In this paper, we introduce $l$-fractional difference operator $\Delta^{(a; l)}$ employing $l$-Pochhammer symbol. We further establish a class of sequence spaces $\mu (\Delta^{(a;l)}, v)$, where $\mu \in \{l_\infty, c_o,  c \}$ which is defined by the use of $l$-fractional difference operator $\Delta^{(a; l)}$. We  study some topological properties and determine $\alpha$, $\beta$ and $\gamma$ duals of the spaces $\mu (\Delta^{(a;l)}, v)$.

Throughout this paper, we use the following symbols: $w$, the set of all complex sequences; $\mathbb{C}$, the set of complex numbers; $\mathbb{R}$, the set of real numbers; $\mathbb{N}$, the set of natural numbers; $\mathbb{N}_0 = \mathbb{N}\cup \{0\}$; $l_{\infty}$, the space of bounded sequences; $c$, the space of convergent sequences; and $c_0$, the space of null sequences.
\section{$l$-fractional difference operator}
In this section, we define $l$-fractional difference operator and study some properties of this operator. Results in this section are equivalent to the results in \cite{baliarsingh2013some}.

\begin{definition}\label{deffractionaldifferenceoperator}
\normalfont
Let $a$ be a real number. Then $l$-fractional difference operator $\Delta^{(a; l)}: w \rightarrow w$ is defined by
\begin{equation}
\Delta^{(a; l)}x  = \Big\{\sum\limits_{i = 0}^{k}\frac{(-a)_{i,  l}}{i!}x_{k - i}\Big\}_k
\end{equation}
for all $x = (x_k) \in w$. For convenience we write the transformation of the $k^{th}$ term of a sequence $x = (x_k)$ by $\Delta^{(a; l)}$ as
 \begin{equation}
\Delta^{(a; l)}x_k  = \sum\limits_{i = 0}^{k}\frac{(-a)_{i,  l}}{i!}x_{k - i}
\end{equation}
\end{definition}

 For some particular values of $a$ and $l$, we have the following observations:
 \begin{itemize}
   \item For $l = 1$, the fractional difference operator $\Delta^{(a; l)}$ is reduced to the operator $\Delta^{(a)}$ (Equation (\ref{baliarsinghfractional})).
   \item If $l =1$ and $a =m$, a positive integer, then  $\Delta^{(a; l)}$ is basically the $m$ - th order backward difference operator $\Delta^{(m)}$, where $\Delta^{(1)}x_k = x_{k} - x_{k-1}$ and $\Delta^{(m)}x_k = \Delta^{(m-1)}(\Delta^{(1)}x_k)$.
   \item $\Delta^{(a; a)}x_k = x_k - a x_{k - 1}$.
   \item For $a = 2l$, we have $\Delta^{(a; l)}x_k = x_k - 2l x_{k -1} + l^2x_{k-2}$.
   \item $\Delta^{(\frac{1}{2}; \frac{1}{4})}x_k = x_k - \frac{1}{2}x_{k -1} + \frac{1}{16}x_{k-2}$.
   \item$\Delta^{(\frac{-1}{2};\frac{1}{4})} = x_k + \frac{1}{2}x_{k-1} + \frac{3}{16}x_{k - 2} + \frac{3}{48}x_{k-3} + \frac{5}{256}x_{k - 4}+ \dots  \\    \hspace{3cm} +\frac{(\frac{-1}{2})(\frac{-1}{2}+\frac{1}{4}) \dots (\frac{-1}{2}+(k-1)\frac{1}{4})}{k!}x_0$.
 \end{itemize}

\begin{theorem}
\normalfont
The $l$-fractional difference operator $\Delta^{(a;l)}$ is a linear operator.
\end{theorem}
\begin{proof}
As the proof of this theorem is straightforward, we omit the proof.
\end{proof}

 \begin{theorem}\label{thcomposition}
 \normalfont
 For real numbers $a$ and $b$, the following results hold:
 \begin{enumerate}
  \item $\Delta^{(a; l)}(\Delta^{(b; l)}x_k) = \Delta^{(a+b; l)}x_k  = \Delta^{(b; l)}(\Delta^{(a; l)}x_k)$
  \item $\Delta^{(a; l)}(\Delta^{(-a; l)}x_k) = x_k$
 \end{enumerate}
  \end{theorem}
 \begin{proof}
 We give the proof of the first result only. The second result can be deduced from the first result. We consider
 \begin{equation}
 \Delta^{(a; l)}(\Delta^{(b; l)}x_k) = \Delta^{(a; l)} \sum\limits_{i = 0}^{k}\frac{(-b)_{i,  l}}{i!}x_{k - i}
 \end{equation}
 As the $l$-fractional difference operator is linear,
 \begin{eqnarray}
  \Delta^{(a; l)}(\Delta^{(b; l)}x_k) &=& \sum\limits_{i = 0}^{k}\frac{(-b)_{i,  l}}{i!}\Delta^{(a; l)}x_{k - i} \nonumber\\
                                                                         &=& \sum\limits_{i = 0}^{k}\frac{(-b)_{i,  l}}{i!}\sum\limits_{j= 0}^{k-i}\frac{(-a)_{j,  l}}{j!}x_{k - i-j}\label{eq:coefficient1}.
  \end{eqnarray}
  Taking the transformation $i+j = m$ or arranging the coefficients of $x_{k-m}$ for $m = 0, 1, \dots , k$, the Equation (\ref{eq:coefficient1}) can be rewritten as
  \begin{equation}\label{eq:bionomial}
   \Delta^{(a; l)}(\Delta^{(b; l)}x_k) =\sum\limits_{m=0}^{k}\Big(\sum\limits_{i=0}^{m}\frac{(-b)_{i,l}}{i!} \frac{(-a)_{(m-i),l}}{(m-i)!}\Big)x_{k-m}
   \end{equation}
   By mathematical induction, it is easy to see that
  \begin{equation}\label{eq:binomial2}
    \frac{(-(a+b))_{m,l}}{m!} = \Big(\sum\limits_{i=0}^{m}\frac{(-b)_{i,l}}{i!} \frac{(-a)_{(m-i),l}}{(m-i)!}\Big).
    \end{equation}
    Then Equations (\ref{eq:bionomial}) and (\ref{eq:binomial2}) imply that
   \begin{equation*}
   \Delta^{(a; l)}(\Delta^{(b; l)}x_k)  =\sum\limits_{m=0}^{k} \frac{(-(a+b))_{m,l}}{m!}x_{k-m} = \Delta^{(a+b; l)}x_k.
  \end{equation*}
 Similarly, we can prove that
  \begin{equation*}
  \Delta^{(b; l)}(\Delta^{(a; l)}x_k) =  \Delta^{(a+b; l)}x_k.
  \end{equation*}
 This proves the theorem.
\end{proof}

 \section{Sequence spaces $\mu (\Delta^{(a;l)}, v)$ for $\mu = l_\infty, c_o,$ and $c$.}
In this section, we introduce a class of sequence spaces $\mu (\Delta^{(a;l)}, v)$, where $\mu \in \{l_\infty, c_o, c\}$, by the use of $l$-fractional difference operator $\Delta^{(a;l)}$. Let $V$ be the set of all real sequences $v = \{v_n\}$ such that $v_n \neq 0$ for all $n \in \mathbb{N}_0$. Then the class of sequence spaces $\mu (\Delta^{(a;l)}, v)$ is defined by
\begin{equation*}
\mu (\Delta^{(a;l)}, v) = \Big\{ x = \{x_n\} \in w : \Big\{ \sum\limits_{j = 0}^{n}v_j \Delta^{(a;l)}x_j  \Big\}_n \in \mu\Big\}.
\end{equation*}
Let $y_n$ denotes the $n^{\text{th}}$ term of the sequence $\Big\{ \sum\limits_{j = 0}^{n}v_j \Delta^{(a;l)}x_j  \Big\}_n$, then
\begin{eqnarray}
y_n &=& \sum\limits_{j = 0}^{n}v_j \Delta^{(a;l)}x_j  \nonumber\\
        &=& \sum\limits_{j = 0}^{n}v_j\Big(\sum\limits_{i = 0}^{j}\frac{(-a)_{i,  l}}{i!}x_{j - i}\Big) \label{eq:coefficient2}
\end{eqnarray}
Taking the transformation $j-i = m$ or arranging the coefficients of $x_m$ for $m = 0, 1, \dots , n$, the Equation (\ref{eq:coefficient2}) can be rewritten as
\begin{equation*}
     y_n   =\sum\limits_{m=0}^{n}\Big(\sum\limits_{i=0}^{n-m}\frac{(-a)_{i,l}}{i!}v_{i+m}\Big)x_m.
\end{equation*}
Now, we consider a matrix $C = (c_{nm})$ such that
\begin{equation}\label{eq:matrixforyn}
c_{nm} = \begin{cases} \sum\limits_{i=0}^{n-m}\frac{(-a)_{i,l}}{i!}v_{i+m}, \ \text{when} \ n\geq m \\ 0, \ \text{when} \ n < m \end{cases}.
\end{equation}
Then the sequence $\{y_n\}$ is $C$-transform of the sequence $\{x_n\}$.

 \begin{theorem}
 \normalfont
The space $\mu(\Delta^{(a;l)}, v)$ for $\mu = l_{\infty}, c$ or $c_0$ is a normed linear space with respect to the norm $\lVert \cdot \rVert_{\mu (\Delta^{(a;l)}, v)}$ defined by
 \begin{equation*}
\lVert x \rVert _{\mu (\Delta^{(a;l)},v)} = \sup_{n}\Big \lvert \sum\limits_{j = 0}^{n}v_j \Delta^{(a;l)}x_j  \Big \rvert.
 \end{equation*}
 \end{theorem}
 \begin{proof}
 We observe that $\lVert \alpha x \rVert _{\mu (\Delta^{(a;l)},v)} = \lvert \alpha \rvert \lVert x \rVert _{\mu (\Delta^{(a;l)},v)} $ for all $\alpha \in \mathbb{C}$ and $\lVert x+y \rVert _{\mu (\Delta^{(a;l)},v)}  \leq \lVert x \rVert _{\mu (\Delta^{(a;l)},v)} + \lVert y \rVert _{\mu (\Delta^{(a;l)},v)} $ for all sequences $x = \{x_n\}$ and $\ y=\{y_n\}$ of $\mu (\Delta^{(a;l)}, v)$. Now suppose that $x=\{x_k\} \in \mu (\Delta^{(a;l)}, v)$ is such that $\lVert x \rVert _{\mu (\Delta^{(a;l)},v)} = 0$. Then
 \begin{equation*}
 \sup_{n}|\sum\limits_{j=0}^{n}v_j \Delta^{(a;l)}x_j|=0.
 \end{equation*}
 This implies that
 \begin{equation*}
 |\sum\limits_{j=0}^{n}v_j \Delta^{(a;l)}x_j|=0 \ \text{for all} \ n \in \mathbb{N}_0.
 \end{equation*}
  This gives,
  \begin{align*}
  v_0 \Delta^{(a;l)}x_0&=0\\
  v_1 \Delta^{(a;l)}x_1& = 0\\
  & \vdots
  \end{align*}
 This shows that $x_n = 0$ for all $n \in \mathbb{N}_0$. That is, $x = (0, 0, \dots)$. Thus $\lVert\cdot \rVert_{\mu (\Delta^{(a;l)},v)}$ is a norm and $\mu (\Delta^{(a;l)}, v)$ is a normed linear space for $\mu = l_{\infty}, c_0$ or $c$.
 \end{proof}

 \begin{theorem}
 \normalfont
 The space $\mu(\Delta^{(a;l)}, v)$ for $\mu = l_{\infty}, c$ or $c_0$ is a complete normed linear space.
 \end{theorem}
 \begin{proof}
 As proof of the theorem for the spaces $\mu(\Delta^{(a;l)}, v)$ for $\mu = l_{\infty}, c$, and $c_0$ run along similar lines, we just prove that the space $l_{\infty}(\Delta^{(a;l)},v)$ is a complete normed linear space. For this, let $\{x^i\}_i = \{x^1 = \{x^1_k\}_{k=0}^{\infty}, x^2 = \{x^2_k\}_{k=0}^{\infty}, \dots\}$ be a Cauchy sequence in $l_{\infty}(\Delta^{(a;l)}, v)$. Then, from definition of Cauchy sequence, there exists a natural number $n_0(\epsilon)$ corresponding to each real number $\epsilon > 0$ such that
 \begin{equation}\label{inq:completeness1}
 \lVert x^i - x^m\rVert_{\mu(\Delta^{(a;l)}, v)} < \epsilon
 \end{equation}
 for all $i, m \geq n_0(\epsilon)$. By definition of the norm $\lVert\cdot \rVert_{\mu (\Delta^{(a;l)},v)}$, Equation (\ref{inq:completeness1}) gives
 \begin{equation}\label{eq:th3.2.2}
 \sup_n \Big \lvert \sum\limits_{j=0}^{n}v_j \Delta^{(a;l)}x^i_j - \sum\limits_{j=0}^{n}v_j \Delta^{(a;l)}x^m_j \Big\rvert < \epsilon
 \end{equation}
  for all $i, m \geq n_0(\epsilon)$. Inequality (\ref{eq:th3.2.2}) can be rewritten with respect to $C$ - transformation (see (\ref{eq:matrixforyn}))as
  \begin{equation*}
  \sup_n \lvert (Cx^i)_n - (Cx^m)_n\rvert < \epsilon
  \end{equation*}
    for all $i, m \geq n_0(\epsilon)$, where $(Cx)_n$ denotes the $n$-th term of the sequence $\{Cx\}$. This implies that for each nonnegative integer $n$
      \begin{equation}\label{inq:3.2.4}
   \lvert (Cx^i)_n - (Cx^m)_n\rvert < \epsilon
  \end{equation}
    for all $i, m \geq n_0(\epsilon)$. Inequality (\ref{inq:3.2.4}) shows that for a fixed nonnegative integer $n$ the sequence $\{(Cx^i)_n\}_i =\{(Cx^1)_n, (Cx^2)_n, \dots\}$ is a Cauchy sequence in the set of real numbers $\mathbb{R}$. As $\mathbb{R}$ is complete, the sequence $\{(Cx^i)_n\}_i$ converges to $(Cx)_n$ (say). Letting $m$ tends to infinity in Inequality (\ref{inq:3.2.4}), we get
    \begin{equation}\label{inq:3.2.5}
   \lvert (Cx^i)_n - (Cx)_n\rvert < \epsilon
  \end{equation}
   for all $i \geq n_0(\epsilon)$. Since Inequality (\ref{inq:3.2.5}) holds for each nonnegative integer $n$, we have
     \begin{equation}\label{inq:3.2.6}
   \sup_n \lvert (Cx^i)_n - (Cx)_n\rvert < \epsilon
  \end{equation}
  for all $i \geq n_0(\epsilon)$, that is
  \begin{equation}\label{inq:3.2.7}
  \lVert x^i - x\rVert_{\mu (\Delta^{(a;l)}, v)} < \epsilon
  \end{equation}
  for all $i \geq n_0(\epsilon)$. Now, it remains to show that the sequence  $x = \{x_n\}$ belongs to $\mu (\Delta^{(a;l)}, v)$. Since the sequence $\{x^i\}$ belongs to $\mu (\Delta^{(a;l)},v)$, there exists a real number $M$ such that
  \begin{equation}\label{inq:3.2.8}
  \sup_n \lvert (Cx^i)_n\rvert \leq M.
  \end{equation}
  With the help of Inequalities (\ref{inq:3.2.6}) and (\ref{inq:3.2.8}), we conclude that
  \begin{equation}
  \sup_n \lvert (Cx)_n\rvert = \sup_n \lvert (Cx)_n - (Cx^i)_n\rvert + \sup_n \lvert (Cx^i)_n \rvert \leq M + \epsilon
  \end{equation}
  This shows that the sequence $x = \{x_n\} \in l_{\infty}(\Delta^{(a;l)}, v)$. Thus, the space $l_{\infty} (\Delta^{(a;l)}, v)$ is a complete normed linear space.
 \end{proof}

  \begin{theorem}
  \normalfont
  The space $\mu(\Delta^{(a;l)}, v)$ (where $\mu = l_{\infty}, c_0$ or $c $) is linearly isomorphic to the space $\mu$.
  \end{theorem}
  \begin{proof}
  To prove this theorem, we have to show that there exists a mapping $T: \mu (\Delta^{(a;l)}, v) \rightarrow \mu$, which is linear and bijective. For our purpose, we suppose that $T$ is defined by
  \begin{equation}
  Tx = \Big\{\sum\limits_{j=0}^{n}v_j \Delta^{(a;l)}x_j\Big\}_n = \{y_n\}
  \end{equation}
  for all $x = \{x_n\} \in \mu (\Delta^{(a;l)}, v)$. We observe that the mapping $T$ is linear and it satisfies the property ``$Tx = \theta \ implies \ x = \theta$", where $\theta = \{0, 0, \dots\}$. The property ``$Tx = \theta \ implies \ x = \theta$" shows that the mapping $T$ is injective. Now it remains to prove that the mapping $T$ is surjective. Therefore, we consider an element $\{y_n\} \in \mu$ and determine a sequence $\{x_n\}$ such that
  \begin{equation}\label{eq:xnforalphadual}
  x_n = \Delta^{(-a;l)}\Big(\frac{y_n - y_{n-1}}{v_n}\Big)
  \end{equation}
  for all $n \in \mathbb{N}_0$. Then we see that
  \begin{equation}\label{eqisomorphic3}
  \sum\limits_{j=0}^{n} v_j \Delta^{(a;l)}x_j = \sum\limits_{j=0}^{n} v_j \Delta^{(a;l)} \Delta^{(-a;l)}\Big(\frac{y_j - y_{j-1}}{v_j}\Big).
  \end{equation}
By Theorem \ref{thcomposition}, the Equation (\ref{eqisomorphic3}) gives
  \begin{equation}
  \sum\limits_{j=0}^{n} v_j \Delta^{(a;l)}x_j = \sum\limits_{j=0}^{n}(y_j - y_{j-1}) = y_n
  \end{equation}
  This shows that $\Big\{\sum\limits_{j=0}^{n} v_j \Delta^{(a;l)}x_j\Big\}_n = \{y_n\} \in \mu$. Using Definition (\ref{deffractionaldifferenceoperator}), we have $\{x_n\}\in \mu(\Delta^{(a;l)}, v)$. Thus we have shown that for every sequence $\{y_n\}\in \mu$, there exists a sequence $\{x_n\} \in \mu(\Delta^{(a;l)}, v)$. That is, the mapping $T$ is a surjective mapping. Hence the space $\mu(\Delta^{(a;l)}, v)$ is linearly isomorphic to the space $\mu$.
  \end{proof}

   \section{$\alpha-, \beta-$ and $\gamma-$ duals of the spaces $\mu (\Delta^{(a;l)}, v)$ for $\mu = l_\infty, c_o,$ and $c$. }
   In this section, we determine $\alpha-, \beta-$ and $\gamma-$ duals of the spaces $\mu (\Delta^{(a;l)}, v)$ for $\mu = l_\infty, c_o$ and $c$. For two sequence spaces $\mu$ and $\lambda$, the multiplier space of $\mu$ and $\lambda$ is defined by
   \begin{equation*}
   S(\mu, \lambda) = \big\{ z = \{z_k\} \in w: zx = \{z_k x_k\} \in \lambda \ \text{for all} \ x = \{x_k\} \in \mu\big\}.
   \end{equation*}
   In particular, $\alpha-$, $\beta-$, and $\gamma-$ duals of a space $\mu$ are defined by $\mu^{\alpha} = S(\mu, l_1)$, $\mu^{\beta} = S(\mu, cs)$ and $\mu^{\gamma} = S(\mu, bs)$ respectively, where $l_1$, $cs$ and $bs$ are the space of absolutely summable sequences, the space of convergent series and the space of bounded series respectively. To find duals of the spaces  $\mu (\Delta^{(a;l)}, v)$, we have used the following lemmas that are given by Stieglitz and Tietz in \cite{stieglitz1977matrixtransformationen}:
   \begin{lemma}\label{lemma1}
   \normalfont
   $B \in (l_{\infty}: l_1) = (c_0: l_1) =(c: l_1)$ if and only if
   \begin{equation}\label{cond:lemma1}
   \sup\limits_{K \in \textit{F}} \sum\limits_{n}  \Big \lvert \sum\limits_{k \in K} b_{nk} \Big\rvert < \infty.
   \end{equation}
   \end{lemma}

   \begin{lemma}\label{lemma2}
   \normalfont
   $B \in (l_{\infty}: c)$ if and only if
   \begin{equation}\label{cond:lemma2,1}
   \lim\limits_{n \rightarrow \infty} b_{nk} \ \text{\normalfont exists for all} \ k  \ \text{\normalfont and}
   \end{equation}
 \begin{equation}\label{cond:lemma2,2}
   \lim\limits_{n}\sum\limits_{k}\lvert b_{nk}\rvert = \sum\limits_{k} \lvert \lim\limits_{n} b_{nk} \rvert.
   \end{equation}
 \end{lemma}

 \begin{lemma}\label{lemma3}
 \normalfont
 $B \in (c:c)$ if and only if (\ref{cond:lemma2,1}),
 \begin{align}
 &\sup_{n} \sum\limits_{k} \lvert b_{nk} \rvert < \infty \ \text{\normalfont and}\label{cond:lemma3,1}\\
 &\lim\limits_{n}\sum\limits_{k}b_{nk} \ \text{\normalfont exists}\label{cond:lemma3,2}.
 \end{align}
 \end{lemma}

 \begin{lemma}\label{lemma4}
 \normalfont
 $B \in (c_0: c)$ if and only if (\ref{cond:lemma2,1}) and (\ref{cond:lemma3,1}) hold.
 \end{lemma}

 \begin{lemma}\label{lemma5}
 \normalfont
 $B \in (l_{\infty}: l_{\infty}) = (c: l_{\infty}) = (c_0: l_{\infty})$ if and only if (\ref{cond:lemma3,1}) holds.
 \end{lemma}

 Now, we determine the duals of the spaces $\mu (\Delta^{(a;l)}, v)$ for $\mu = l_\infty, c_o,$ and $c$ through following theorem.
 \begin{theorem}
 \normalfont
 Let $D = (d_{nk})$ and $E = (e_{nk})$ are two matrices such that
 \begin{equation}
 d_{nk} = \begin{cases}
                    0; \ \hspace{4.5cm}(k>n) \\
                    \frac{z_k}{v_k};  \hspace{4.5cm}(k = n)\\
                    z_k\left( \frac{(a)_{n-k, l}}{(n-k)! v_k} - \frac{(a)_{n-k-1, l}}{(n-k-1)! v_{k+1}}\right); \ (k<n)
                    \end{cases}
                    \text{and}
 \end{equation}
 \begin{equation}
 e_{nk} = \begin{cases}
                   0; \ \hspace{6.5cm} (k>n) \\
                   \frac{z_k}{v_k};  \hspace{6.5cm}(k = n)\\
                   \frac{1}{v_k}\sum\limits_{i=0}^{n-k} \frac{(a)_{i,l}}{i!}z_{k+i} - \frac{1}{v_{k+1}}\sum\limits_{i=0}^{n-k-1}\frac{(a)_{i,l}}{i!}z_{k+i+1}; \ (k<n)
                   \end{cases}
 \end{equation}
 and consider the sets $A_1, A_2, A_3, A_4$ and $A_5$ which are defined as follows:
 \begin{eqnarray*}
 A_1&=& \Big\{ z = (z_k) \in w :  \sup\limits_{K \in \textit{F}} \sum\limits_{n}  \Big \lvert \sum\limits_{k \in K} d_{nk} \Big\rvert < \infty \Big\}, \\
 A_2&=& \Big\{ z = (z_k) \in w : \lim\limits_{n \rightarrow \infty} e_{nk} \ \text{\normalfont exists for all} \ k \Big\},\\
 A_3&=& \Big\{ z = (z_k) \in w : \lim\limits_{n}\sum\limits_{k}\lvert e_{nk}\rvert = \sum\limits_{k} \lvert \lim\limits_{n} e_{nk} \rvert\Big\},\\
 A_4&=& \Big\{ z = (z_k) \in w : \sup_{n} \sum\limits_{k} \lvert e_{nk} \rvert < \infty\Big\} \ \text{\normalfont and}\\
 A_5&=& \Big\{ z = (z_k) \in w :  \lim\limits_{n}\sum\limits_{k}e_{nk} \ \text{\normalfont exists} \Big\}.
 \end{eqnarray*}
 Then,
 \begin{align*}
 &\big\{l_{\infty}(\Delta^{(a;l)}, v)\big\}^{\alpha} = A_1, \  \big\{l_{\infty}(\Delta^{(a;l)}, v)\big\}^{\beta} = A_2 \cap A_3, \   \big\{l_{\infty}(\Delta^{(a;l)}, v)\big\}^{\gamma} = A_4\\
 &\big\{c_{0}(\Delta^{(a;l)}, v)\big\}^{\alpha} = A_1, \  \big\{c_{0}(\Delta^{(a;l)}, v)\big\}^{\beta} = A_2 \cap A_4, \   \big\{c_{0}(\Delta^{(a;l)}, v)\big\}^{\gamma} = A_4\\
 & \big\{c(\Delta^{(a;l)}, v)\big\}^{\alpha} = A_1, \  \big\{c(\Delta^{(a;l)}, v)\big\}^{\beta} = A_2 \cap A_4 \cap A_5  \  \text{\normalfont and} \ \big\{c(\Delta^{(a;l)}, v)\big\}^{\gamma} = A_4.
 \end{align*}
 \end{theorem}
\begin{proof}
We determine duals of the space $c_0(\Delta^{(a;l)},v)$. Duals of other spaces $l_{\infty}(\Delta^{(a;l)},v)$ and $c(\Delta^{(a;l)},v)$ can be deduced on the similar lines as that of the $c_0(\Delta^{(a;l)},v)$. By Equation (\ref{eq:xnforalphadual}), we obtain that
\begin{equation}
z_nx_n = z_n  \Delta^{(-a;l)}\Big(\frac{y_n - y_{n-1}}{v_n}\Big) = (Dy)_n
\end{equation}
for all $\{x_n\} \in c_0(\Delta^{(a;l)},v)$ and $\{y_n\} \in c_0$. Now the sequence $\{z_n\} \in \big\{c_0(\Delta^{(a;l)},v)\big \}^{\alpha}$ iff $\{z_n x_n\} \in l_1$ for all $\{x_n\} \in c_0(\Delta^{(a;l)},v)$. That is, $\{z_n\} \in \big\{c_0(\Delta^{(a;l)},v)\big \}^{\alpha}$ iff $Dy \in l_1$ for all $\{y_n\} \in c_0$. In other words, $\{z_n\} \in \big\{c_0(\Delta^{(a;l)},v)\big \}^{\alpha}$ iff $D \in (c_0: l_1)$. Keeping this in view and using Lemma \ref{lemma1}, we conclude that
\begin{equation}
\big\{c_0(\Delta^{(a;l)},v)\big \}^{\alpha} = A_1
\end{equation}
To find $\beta$-dual, we define a squence $\{t_n\}$ as follows:
\begin{equation}
t_n = \sum\limits_{r=0}^{n} z_r x_r = (Ey)_n.
\end{equation}
Then the sequence $\{z_n\} \in \big\{c_0(\Delta^{(a;l)},v)\big \}^{\beta}$ iff $\{t_n\} \in c$ for all $\{x_n\}\in c_0(\Delta^{(a;l)},v)$. That is, $\{z_n\} \in \big\{c_0(\Delta^{(a;l)},v)\big \}^{\beta}$ iff $Ey \in c$ for all $\{y_n\} \in c_0$. In other words, $\{z_n\} \in \big\{c_0(\Delta^{(a;l)},v)\big \}^{\beta}$ iff $E \in (c_0: c)$. Keeping this in view and using Lemma \ref{lemma4}, we conclude that
\begin{equation}
\big\{c_0(\Delta^{(a;l)},v)\big \}^{\beta} = A_2 \cap A_4.
\end{equation}
Similarly, using Lemma \ref{lemma5}, we get
\begin{equation}
\big\{c_0(\Delta^{(a;l)},v)\big \}^{\gamma} =  A_4.
\end{equation}
\end{proof}
\section{Conclusions}
By introducing $l$- fractional difference operator, we have given a more general fractional order difference operator. Also, we have defined a class of difference sequence spaces with the help of $l$- fractional difference operator. Further, we have studied some topological properties and determined the duals of the spaces.

\bibliographystyle{plain}
\bibliography{mybib}

\end{document}